\documentclass[12pt]{amsart}  
\usepackage{amsmath,amstext,amsthm,amssymb,amsxtra,graphicx}
\usepackage[colorlinks,citecolor=red,pagebackref,hypertexnames=true]{hyperref}
\usepackage[backrefs,msc-links,nobysame]{amsrefs}

\allowdisplaybreaks

\theoremstyle{plain} % definition 
\newtheorem{lemma}{Lemma}

\newtheorem{theorem}{Theorem}
\newtheorem{corollary}{Corollary}
\theoremstyle{definition}

\theoremstyle{remark}

\newcommand{\beql}[1]{\begin{equation}\label{#1}}
\newcommand{\eeq}{\end{equation}}
\newcommand{\comment}[1]{}

\newcommand{\Abs}[1]{{\left|{#1}\right|}}

\newcommand{\Set}[1]{{\left\{{#1}\right\}}}

\newcommand{\RR}{{\mathbb R}}
\newcommand{\CC}{{\mathbb C}}
\newcommand{\ZZ}{{\mathbb Z}}

\newcommand{\inner}[2]{{\langle #1, #2 \rangle}}

\newcommand{\supp}{{\rm supp\,}}
\newcommand{\esssup}{{\rm esssup\,}}
\newcommand{\diam}{{\rm diam\,}}

\newcommand{\ft}[1]{\widehat{#1}}

\begin{document}
%%%%%%%%%%%%%%%%%%%%%%%%%%%%%  Title
\title[]{Discrepancy of line segments for general lattice checkerboards}

\author[M. Kolountzakis]{{Mihail N. Kolountzakis}}
\address{M.K.: Department of Mathematics and Aplied Mathematics, University of Crete, Voutes Campus., GR-700 13, Heraklion, Greece}
\email{kolount@gmail.com}

\thanks{M.K.\ has been partially supported by the ``Aristeia II'' action (Project
FOURIERDIG) of the operational program Education and Lifelong Learning
and is co-funded by the European Social Fund and Greek national resources.}

\date{January 9, 2016}

\begin{abstract}
In a series of papers recently 
``checkerboard discrepancy'' has been introduced, where a black-and-white checkerboard background
induces a coloring on any curve, and thus a discrepancy, i.e., the difference of
the length of the curve colored white and the length colored black.  Mainly
straight lines and circles have been studied and the general situation is that,
no matter what the background coloring, there is always a curve in the family
studied whose discrepancy is at least of the order of the square root of the
length of the curve.

In this paper we generalize the shape of the background, keeping the lattice
structure. Our background now consists of lattice copies of any bounded
fundamental domain of the lattice, and not necessarily of squares, as was the
case in the previous papers.  As the decay properties of the Fourier Transform of the
indicator function of the square were strongly used before, we now have to use
a quite different proof, in which the tiling and spectral properties of the
fundamental domain play a role.
\end{abstract}

\maketitle
\tableofcontents

\section{Introduction}

\subsection{Checkerboard discrepancies}

In the series of papers \cites{K, IK, KP} the discrepancies of curves against a
two-colored checkerboard background were introduced and upper and lower bounds
for straight lines and circles were given.  Briefly, the discrepancy of a curve
$\Gamma$ against a two-colored background, say a black-and-white background, is the
absolute difference between the lengths of the black and the white part of $\Gamma$.
See Figure \ref{fig:board}.

The motivating question was the following.
\begin{quotation}
Can we color the usual lattice subdivision of the plane into a checkerboard
using two colors so that any line segment placed on the plane has {\em bounded}
discrepancy?
\end{quotation}
It was shown in \cites{K} that the answer is negative: for any two-coloring of
the infinite checkerboard, i.e., of the partition
$$
\RR^2 = \bigcup_{k \in \ZZ^2} \left( k+[0,1)^2 \right),
$$
there are arbitrarily large line segments $I$ whose discrepancy is at least $C
\sqrt{\Abs{I}}$, where $\Abs{I}$ is the length of the segment $I$ and $C$ is an
absolute constant.  Further questions, such as constructions of colorings that
give as small discrepancy as possible, as well as the discrepancy of other
shapes, mainly circles and circular arcs, have been dealt with in \cites{K, IK,
KP}.

\begin{figure}[h]
\begin{center} \resizebox{5cm}{!}{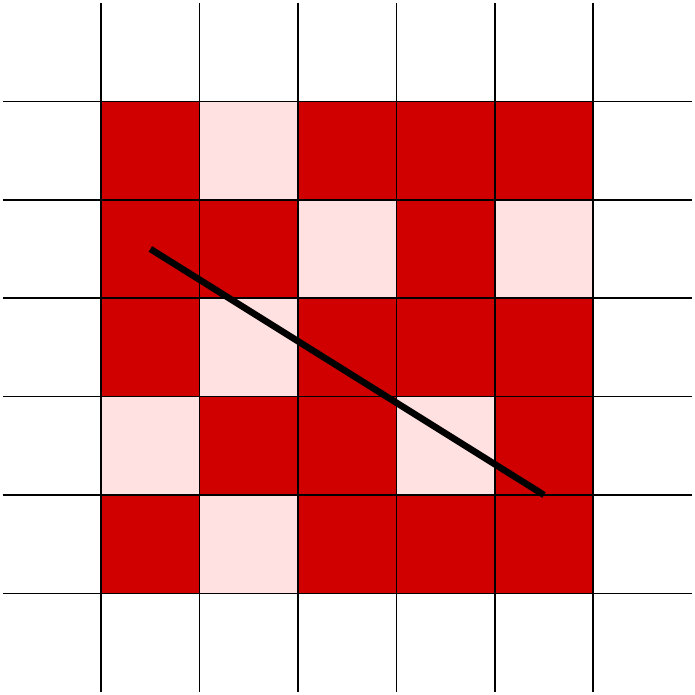} \end{center}
\caption{A two-colored checkerboard with a line segment on it.}
\label{fig:board}
\end{figure}

\subsection{Geometric discrepancy}
The question dealt with in this paper as well as in \cites{K, IK, KP}
falls naturally into the subject of
geometric discrepancy \cites{alexander,beck-chen,chazelle,drmota,matousek}.  In
this research area there is usually an underlying measure $\mu$ as well as a
family ${\mathcal F}$ of allowed subsets of Euclidean space, on which the
measure $\mu$ is evaluated and upper and lower bounds are sought on the range
of $\mu$ on ${\mathcal F}$.  The most classical case is that where $\mu$ is a
normalized collection of points masses in the unit square minus Lebesgue
measure and ${\mathcal F}$ consists of all axis-aligned rectangles in the unit
square.  Usually the underlying measure $\mu$ has an atomic part (point masses)
and the family ${\mathcal F}$ consists of ``fat'' sets.  In the problem we are
studying here the measure $\mu$ has no atomic part (it is absolutely
continuous) and the collection ${\mathcal F}$ consists of all straight line
segments, or other curves, which may be considered thin sets, and, strictly
speaking, $\mu$ is $0$ on these sets.  The measure $\mu$ is however regular
enough to induce a measure on the one-dimensional objects whose discrepancy we
study, and it is of this induced measure that we study the range.

We note that discrepancies with respect to non-atomic colorings have been
considered by Rogers in \cite{RO1}, \cite{RO2} and \cite{RO3} where the author
considers, among other things, the discrepancy of lines and half spaces with
respect to finite colorings of the plane. Rogers proves lower bounds for the
discrepancy of these families of sets with respect to generalized colorings.
His results do not seem to be comparable to the results in this paper or those in \cites{K, IK, KP}.

\subsection{New results}
In this paper we generalize the results of \cites{K} to more general two-colorings
of the plane. We keep the lattice structure of the original problem but now
we tile the plane with much more general shapes than squares.
For a lattice $T \subseteq \RR^2$
(discrete, additive subgroup of $\RR^2$ whose linear span is the plane)
we consider a fundamental domain $Q$ of $T$, i.e., a set containing exactly one
element from each coset of $T$ in $\RR^2$.
This is the same as asking that
$$
\RR^2 = \bigcup_{t \in T} (t+Q)\ \ \ \text{is a partition of $\RR^2$.}
$$
We then color each set $t+Q$ in this partition using using one of two possible colors and this defines the
discrepancy of every line segment on the plane on which the coloring function is measurable.
This happens for almost all (Lebesgue) straight lines normal to any given direction by Fubini's theorem.

See Figures \ref{fig:lboard} and \ref{fig:hboard} for two examples of such generalized
checkerboards.

\begin{figure}[h]
 \begin{center} \resizebox{5cm}{!}{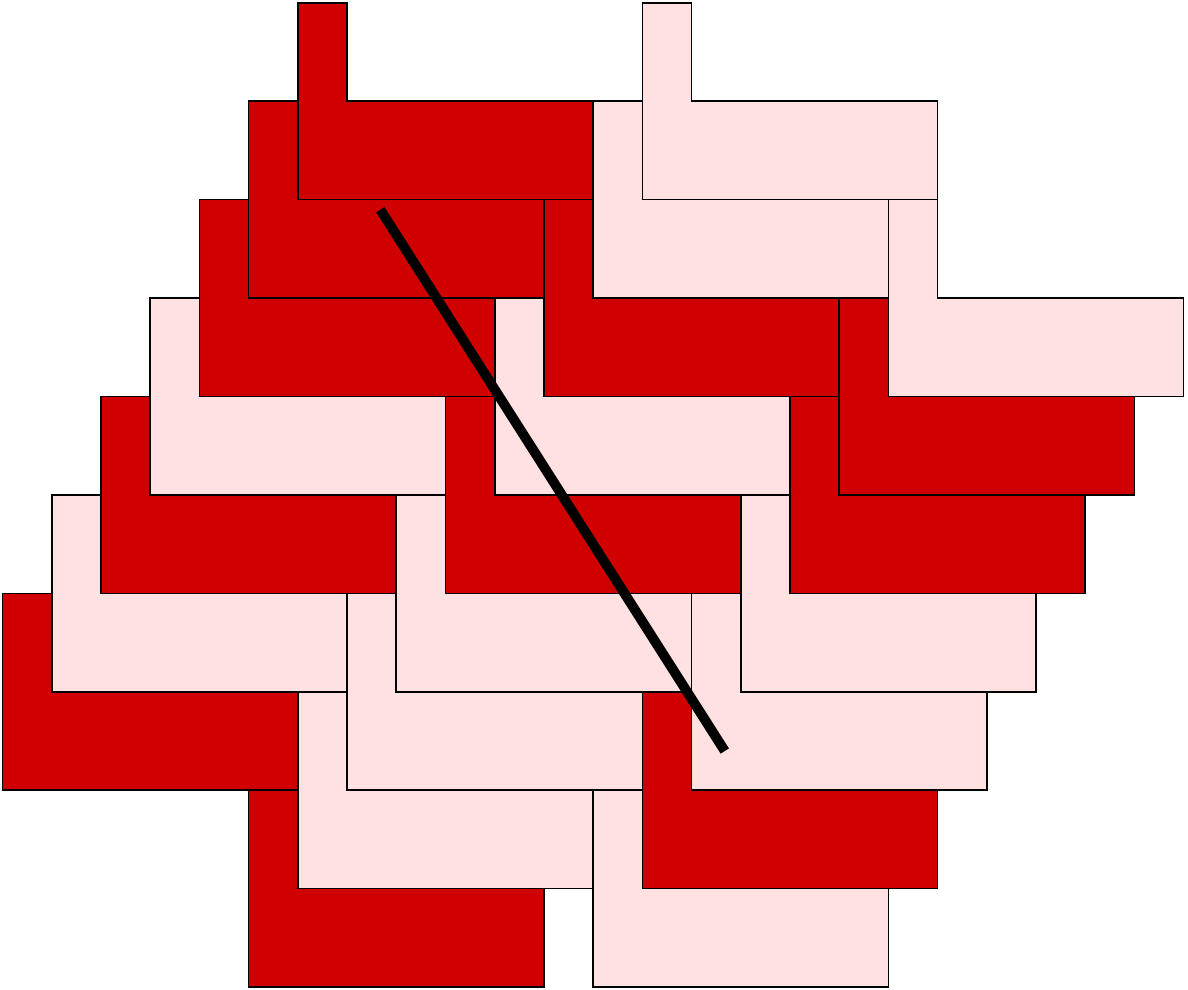} \end{center}
\caption{A two-colored $L$-shaped checkerboard.}
\label{fig:lboard}
\end{figure}

\begin{figure}[h]
 \begin{center} \resizebox{5cm}{!}{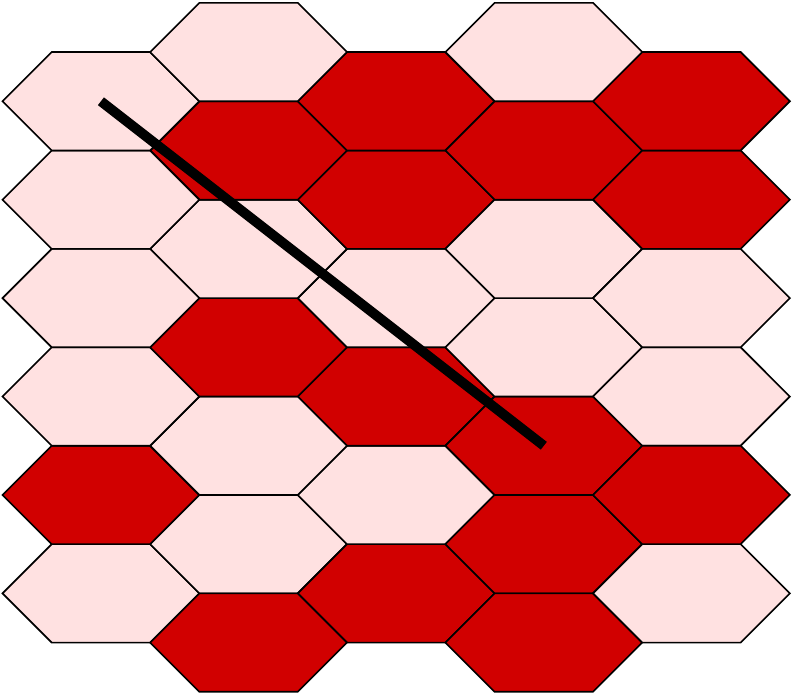} \end{center}
\caption{A two-colored hexagon-shaped checkerboard.}
\label{fig:hboard}
\end{figure}

The following Theorem \ref{th:line-discrepancy} and Corollary \ref{cor:full-plane} are our main results.
The method generalizes to the results in
\cites{IK, KP} as well, but we leave them out to keep the exposition simple.

%%%%%%%%%%%%%%%%%%%%%%%%%%%%%%%%%%%%%%%%%%%%%%%%%%%%
\begin{theorem}\label{th:line-discrepancy}
Suppose $T \subseteq \RR^2$ is a lattice of area 1 and let $Q$
be a bounded, Lebesgue measurable fundamental domain of $T$.

Let also $G \subseteq T$ be
a finite subset of lattice points
and let the function $f:\RR^2 \to \CC$ (the {\em coloring}) be defined by
$$
f(x) = \sum_{g \in G} z_g \chi_Q(x-g),
$$
for some complex numbers $z_g$, $g \in G$.

Then there is a straight line $S$ such that
$f$ is measurable on $S$ and
$$
\Abs{\int_S f} \ge C (\diam{G})^{-1/2} \left(\sum_{g \in G} \Abs{z_g}^2\right)^{1/2}, 
$$
where $C$ is a positive constant that may depend on $T$ and $Q$ only.
\end{theorem}

\noindent
{\bf Notation:} The letter $C$ will always denote a positive number (whose possible dependence on
parameters will be clearly stated) which is not the same in all its occurences.

%%%%%%%%%%%%%%%%%%%%%%%%%%%%%%%%%%%%%%%%%%%%%%%%%%%%
The following speaks more directly about two-colorings of the plane.
\begin{corollary}\label{cor:full-plane}
Suppose $T$ and $Q$ satisfy the assumptions of Theorem
\ref{th:line-discrepancy}.
Then there is a positive constant $C$, which may depend
on $Q$ and $T$ only, such that for any assignment
$z_t = \pm 1$, for $t \in T$, and corresponding coloring function
$$
f(x) = \sum_{t \in T} z_g \chi_Q(x-t)
$$
there are arbitrarily large line segments $I$ such that
$f$ is measurable on the straight line supporting $I$ and
$$
\Abs{\int_I f} \ge C \sqrt{\Abs{I}}.
$$
\end{corollary}

\begin{proof}[Proof of Corollary \ref{cor:full-plane}]
Let $R>0$ be large and let $G \subseteq T$ consist of all $t \in T$ such that
$$
(t + Q) \cap [0, R]^2 \neq \emptyset.
$$
Since $Q$ is bounded the set $G$ is finite.
By the boundedness of $Q$ we have that
\beql{bound1}
\diam{G} = \sqrt{2} R + O(1),
\eeq
as $R \to \infty$.
For the same reason we have that
\beql{bound2}
\Abs{G} = R^2 + O(R).
\eeq
The lower order terms in \eqref{bound1} and \eqref{bound2} are due to the boundary.

We may now apply Theorem \ref{th:line-discrepancy}
to the set $G$ and the coloring function
$$
F(x) = \sum_{g \in G} z_g \chi_Q(x-g).
$$
For sufficiently large $R$ we
obtain a straight line $S$ such that $F$ is measurable on $S$ and
\begin{align*}
\Abs{\int_S F} &\ge C (\diam{G})^{-1/2} \Abs{G}^{1/2} \\
 &\ge C (R+O(1))^{-1/2} (R^2+O(R))^{1/2} \\
 &\ge C \sqrt{R}.
\end{align*}
We may partition
$$
S \cap \supp{F} = I \cup E,
$$
where $I$ is the line segment $S \cap [0,R]^2$ and the one-dimensional measure of the set $E$ is at most a constant.
It follows that
\begin{align*}
\Abs{\int_I f} &= \Abs{\int_I F}\\
 &\ge \Abs{\int_S F} - O(1)\\
 &\ge C \sqrt{R}\\
 &\ge C \sqrt{\Abs{I}},
\end{align*}
as we had to show.
\end{proof}

\subsection{Does the checkerboard pattern matter?}
It is a natural question whether the checkerboard pattern of the coloring function really
matters for the discrepancy lower bounds.
Could it be the case that for {\em any} measurable coloring function
$$
f:[0,R]^2 \to \Set{-1, +1}
$$
there is always a straight line $S$, on which the restriction of $f$ is measurable and
$$
\Abs{\int_S f} \ge C \sqrt{R},
$$
as is the case when the function $f$ is a checkerboard coloring of $1\times 1$ squares?

That this is not so can be proved using a (randomized) construction of \cites{K}.
In \cites{K} it was proved that for every $N$ there is a $\pm 1$ coloring of the
usual subdivision in unit squares of $[0,N]^2$ such that for any straight line segment $I$
we have
$$
\Abs{\int_I f} \le C \sqrt{N \log N}.
$$
Take such a coloring of the square $[0, N]^2$. The discrepancy of any segment is at most $C\cdot \sqrt{N \log N}$.
Now scale the coloring down to the square
$$
\left[0,\sqrt{\frac{N}{\log N}}\right]^2,
$$
so that it becomes a coloring of that square
with a square checkerboard of side-length $\frac{1}{\sqrt{N \log N}}$. The discrepancy scales proportionally so any line
segment in this new coloring now has discrepancy at most $C$.

%%%%%%%%%%%%%%%%%%%%%%%%%%%%%%%%%%%%%%%%%%%%%%%%%%%%%%%%%%%%%%%%%%
%%%%%%%%%%%%%%%%%%%%%%%%%%%%%%%%%%%%%%%%%%%%%%%%%%%%%%%%%%%%%%%%%%
%%%%%%%%%%%%%%%%%%%%%%%%%%%%%%%%%%%%%%%%%%%%%%%%%%%%%%%%%%%%%%%%%%
\section{Proof of the Main Theorem}

\begin{lemma}\label{lm:tail}
Suppose $L>0$ is a constant,
$\Lambda \subseteq \RR^d$ is a lattice and $g \in L^1(\RR^d)$ is a nonnegative function,
uniformly continuous on $\RR^d$,
such that
\beql{tiling}
\sum_{t \in \Lambda} g(x-t) = L,
\eeq
for all $x \in \RR^d$.
Then for each $\epsilon>0$ and for each $x \in \RR^d$ there is $r=r(x, \epsilon)>0$ and $R = R(x, \epsilon)>0$ such that
\beql{tail}
\sum_{t \in \Lambda, \Abs{t}>R} g(y-t) < \epsilon,\ \ \ \text{ whenever } \Abs{y-x}<r.
\eeq
\end{lemma}

\begin{proof}
For each $x \in \RR^d$ it follows from \eqref{tiling} that there is $R = R(x, \epsilon) > 0$ such that
$$
\sum_{t \in \Lambda, \Abs{t}>R} g(x-t) < \epsilon/2.
$$
Let $N = \Set{t\in \Lambda: \Abs{t} \le R}$. It follows that
\beql{tmp1}
\sum_{t \in N} g(x-t) > L - \epsilon/2.
\eeq
This is a finite sum.
Since $g$ is uniformly continuous we can select $r = r(x, \epsilon)>0$ such that
$$
\Abs{x-y}<r \Rightarrow \Abs{g(x)-g(y)}< \frac{\epsilon}{2\Abs{N}}.
$$
Combining this with \eqref{tmp1} we obtain that whenever $\Abs{x-y}<r$ we have
$$
\sum_{t \in N} g(y-t) > L - \epsilon,
$$
and, because of \eqref{tiling}, we have that \eqref{tail} holds.
\end{proof}

\begin{corollary}\label{cor:compact-tail}
With the assumptions of Lemma \ref{lm:tail}, if $B \subseteq \RR^d$ is a bounded set
then
for each $\epsilon>0$ there is $R>0$ such that
\beql{compact-tail}
\sum_{t \in \Lambda, \Abs{t}>R} g(x-t) < \epsilon,\ \ \ \text{ whenever } x \in B.
\eeq
\end{corollary}

\begin{proof}
Let $\epsilon>0$.
The closure $\overline{B}$ of $B$ is compact and is covered by the union of open balls
$$
\bigcup_{x \in \overline{B}} B_{r(x, \epsilon)}(x),
$$
where the radius $r(x, \epsilon)$ is given by Lemma \ref{lm:tail}.
Therefore there exist $x_1, \ldots, x_n \in \overline{B}$ such that
$$
\overline{B} \subseteq \bigcup_{j=1}^n B_{r(x_j, \epsilon)}(x_j).
$$
Let $R = \max\Set{R(x_1, \epsilon), R(x_2, \epsilon), \ldots, R(x_n, \epsilon)}$,
where the numbers $R(x_j, \epsilon)$ are those given by Lemma \ref{lm:tail},
and the conclusion follows.
\end{proof}

\begin{proof}[Proof of Theorem \ref{th:line-discrepancy}]

The definition of the Fourier Transform that we use is
$$
\ft{f}(\xi) = \int_{\RR^d} f(x) e^{-2\pi i \xi\cdot x}\,dx
$$
for $f \in L^1(\RR^d)$ and $\xi \in \RR^d$.

For a straight line $L$ through the origin let us denote by $\pi_L f$ (the projection of $f$ onto $L$)
the function of $t \in \RR$ given by
$$
\pi_Lf (t) = \int_{\RR} f(tu+su^{\perp}) \,ds,
$$
where $u$ is a unit vector along $L$ and $u^\perp$ is a unit vector orthogonal to $u$.
By Fubini's theorem
for any straight line $L$ through the origin $\pi_Lf (t)$ is well defined for almost all (Lebesgue) values
of $t \in \RR$.

It is well known and easy to see that the Fourier Transform of $\pi_Lf$ is equal to the restriction
of the (two-dimensional) Fourier Transform of $f$, $\ft f$, on $L$:
$$
\ft{\pi_Lf}(\xi) = \ft{f}(\xi u),\ \ \xi\in\RR.
$$
Write
$$
M = \esssup_{L,t} \Abs{\pi_Lf(t)}
$$
where the essential supremum is taken over all lines $L$ through the origin and real numbers $t$.

For the Fourier Transform of $f = \chi_Q * \sum_{g \in G} z_g \delta_g$ we have
\beql{f-equals}
\ft{f}(\xi) = \ft{\chi_Q}(\xi) \phi(\xi),
\eeq
where the trigonometric polynomial
$$
\phi(\xi) = \sum_{g \in G} z_g e^{-2\pi i g\xi}
$$
is a $T^*$-periodic function.
Here $T^* = \Set{x\in\RR^2:\ \inner{x}{t} \in \ZZ,\  \forall t \in T}$ is the dual lattice of $T$, which also has area 1.
It is also true that $T^{**} = T$.

If $B$ is any measurable fundamental domain of $T^*$ we have
that $L^2(B)$ has the functions
$$
e_{t}(x) = e^{-2\pi i t\cdot x},\ \ \ t \in T,
$$
as an orthonormal basis. This is merely a restatement, after a linear transformation, of the fact that
the exponentials with integer frequencies form an orthonormal basis of $L^2([0,1]^2)$.
Parseval's identity then gives
\beql{phi-squared-int}
\int_B \Abs{\phi(\xi)}^2 = \sum_{g \in G} \Abs{z_g}^2.
\eeq

We are going to show that
$$
M>C (\diam G)^{-1/2} \left(\sum_{g \in G} \Abs{z_g}^2\right)^{1/2},
$$
for some constant $C>0$ that may depend only on the lattice $T$ and the chosen fundamental domain $Q$.

Write $D = \diam G$. It follows that $\diam{\supp f}\le D+C$, for some constant $C$ that depends
on the fundamental domain $Q$.
Since $Q$ has area 1 it follows that $D$ is bounded below by a constant and, therefore, we have that
$$
\diam{\supp f} \le C D.
$$
As we also have
$$
\diam{\supp{\pi_L f}} \le C D
$$
and $\Abs{\pi_Lf(t)} \le M$ for almost all $t \in \RR$
we obtain from Parseval's equality that
\beql{l2-line}
\int_{\RR} \Abs{\ft{f}(tu)}^2\,dt = \int_{\RR} \Abs{\pi_Lf(t)}^2\,dt \le CM^2 D.
\eeq
It also follows from Parseval's equality that
\beql{l2-plane}
\sum_{g \in G} \Abs{z_g}^2 = \int_{\RR^2} \Abs{f}^2 = \int_{\RR^2}\Abs{\ft{f}}^2.
\eeq

We now make the observation that since $Q$ tiles the plane with the lattice $T$ it follows
that $Q$ has the lattice $T^*$ as a {\em spectrum} (see e.g. \cites{Milano}).
In other words, the family of exponentials
\beql{t-star-exponentials}
e_{t^*}(x) = e^{2\pi i t^*\cdot x},\ \ \ t^* \in T^*,
\eeq
forms an orthonormal basis for $L^2(Q)$.
This implies the tiling condition
\beql{ps-tiling}
\sum_{t^* \in T^*} g(\xi-t^*) = 1,\ \ \text{ for all } \xi \in \RR^2,
\eeq
where
$$
g(\xi) = \Abs{\ft{\chi_Q}(\xi)}^2 \text{ and } \int g = 1.
$$
One can easily prove \eqref{ps-tiling} by observing that the inner product, in $L^2(Q)$, of the function
$e_\xi(t) = e^{2\pi i \xi\cdot x}$ and the function $e_{t^*}(x) = e^{2\pi i t^*\cdot x}$ is equal to
$$
\ft{\chi_Q}(\xi-t^*)
$$
and then rewriting Parseval's identity for the function $e_\xi(t)$ with respect to
the complete orthonormal system \eqref{t-star-exponentials}.
The function $g(x)$ is uniformly continuous in $\RR^2$ (since $\ft{\chi_Q}$ is both uniformly continuous
and bounded), therefore Lemma \ref{lm:tail} and Corollary \ref{cor:compact-tail} are applicable
to the function $g$ and the lattice $\Lambda=T^*$.

Taking $B$ to be a bounded fundamental domain of $T^*$ and $\epsilon=1/2$ we have therefore from Corollary \ref{cor:compact-tail} that 
there exists a finite $R>0$ such that for $\xi \in B$ we have
\beql{tail-bound}
\sum_{t^* \in T^*, \Abs{t^*} > R} g(\xi -t^*) < \frac{1}{2}.
\eeq
Let now $R'>0$ be such that
\beql{containment}
\Set{\Abs{\xi} > R'} \subseteq \bigcup_{t^* \in T^*, \Abs{t^*} > R} (t^*+B).
\eeq
This exists because $B$ is a bounded set.
We may assume that $R'$ depends on $T$ alone, by taking the fundamental domain $B$ of $T^*$ that minimizes $R'$.

We have
\begin{align*}
\int_{\Abs{\xi} > R'} \Abs{\ft{f}(\xi)}^2
 &= \int_{\Abs{\xi} > R'} g(\xi) \Abs{\phi(\xi)}^2 \text{\ \ \  (from \eqref{f-equals})}\\
 &\le  \sum_{t^* \in T^*, \Abs{t^*} > R} \int_{t^*+B} g(\xi) \Abs{\phi(\xi)}^2 \text{\ \ \  (from \eqref{containment})}\\
 &=  \sum_{t^* \in T^*, \Abs{t^*} > R} \int_{B} g(\xi+t^*) \Abs{\phi(\xi)}^2 \text{\ \ \  (since $\phi(\xi)$ is $T^*$-periodic)}\\
 &=  \int_B \left( \sum_{t^* \in T^*, \Abs{t^*} > R} g(\xi+t^*) \right) \Abs{\phi(\xi)}^2 \\
 &\le \frac{1}{2} \int_B \Abs{\phi(\xi)}^2 \text{\ \ \  (from \eqref{tail-bound})}\\
 &= \frac{1}{2} \sum_{g \in G} \Abs{z_g}^2 \text{\ \ \  (from \eqref{phi-squared-int})}.
\end{align*}
Writing $S^1$ for the unit circle in $\RR^2$
it follows from \eqref{l2-plane} and the inequality just proved that
\begin{align*}
\frac{1}{2} \sum_{g \in G} \Abs{z_g}^2
 &\le \int_{\Abs{\xi} \le R'} \Abs{\ft{f}(\xi)}^2\\
 &= \int_{0 \le t \le R'} \int_{u \in S^1} t \Abs{\ft{f}(tu)}^2 \,du \, dt \text{\ \ \  (in polar coordinates)}\\
 &\le R' \int_{u \in S^1} \int_{\RR} \Abs{\ft{f}(tu)}^2 \,dt \,du\\
 &\le R' 2\pi C M^2 D \text{\ \ \ (from \eqref{l2-line})}.
\end{align*}
Solving for $M$ in the inequality above, and remembering that $R'$ depends only on $Q$ and $T$, gives
$$
M \ge C \frac{1}{\sqrt{D}} \left( \sum_{g \in G} \Abs{z_g}^2 \right)^{1/2},
$$
for $C$ depending on $Q$ and $T$ only.
\end{proof}

\end{document}